\pdfoutput=1 %
\documentclass[reqno]{amsart} %

\usepackage[utf8]{inputenc}

\usepackage{braket}
\usepackage{amsmath}
\usepackage{amstext}
\usepackage{amssymb}
\usepackage{amsthm}
\usepackage{mathtools}
\usepackage{stmaryrd}
\usepackage{mathrsfs}
\usepackage{extarrows}
\usepackage{afterpage}
\usepackage{tabu}
\usepackage{array}
\usepackage{microtype}
\usepackage{graphicx}
\usepackage{amsfonts}
\usepackage{units}
\usepackage[shortlabels]{enumitem}
\usepackage{url}
\usepackage{bm}
\usepackage{soul}
\usepackage{tikz}
\usetikzlibrary{matrix,arrows,calc,fit,cd,positioning,intersections,arrows.meta,backgrounds}
\usepackage{tikz-3dplot}

\usepackage{cleveref}
\usepackage[margin=1cm]{caption}
\usepackage[margin=0.5cm]{subcaption}
\usepackage{colortbl}

\usepackage{fullpage}

\theoremstyle{definition}
\newtheorem{theorem}{Theorem}[section]

\newtheorem{lemma}[theorem]{Lemma}

\newtheorem{conjecture}[theorem]{Conjecture}
\newtheorem{question}[theorem]{Question}

\newcommand{\N}{\mathbb{N}}
\newcommand{\Z}{\mathbb{Z}}
\newcommand{\R}{\mathbb{R}}
\newcommand{\C}{\mathbb{C}}

\renewcommand{\H}{\mathbb{H}}

\renewcommand{\S}{\mathbb{S}}

\newcommand{\A}{\mathcal{A}}

\DeclareMathOperator{\id}{id}

\DeclareMathOperator{\im}{im}

\newcommand{\<}{\langle}
\renewcommand{\>}{\rangle}

\DeclareMathOperator{\Mov}{\textsc{Mov}}

\DeclareMathOperator{\Isom}{\textsc{Isom}}

\title{The dual approach to the $K(\pi, 1)$ conjecture}
\author{Giovanni Paolini}

\begin{document}

\begin{abstract}
	Dual presentations of Coxeter groups have recently led to breakthroughs in our understanding of affine Artin groups. In particular, they led to the proof of the $K(\pi, 1)$ conjecture and to the solution of the word problem.
	Will the ``dual approach'' extend to more general classes of Coxeter and Artin groups?
	In this paper, we describe the techniques used to prove the $K(\pi, 1)$ conjecture for affine Artin groups and we ask a series of questions that are mostly open beyond the spherical and affine cases.
\end{abstract}

\maketitle

The $K(\pi, 1)$ conjecture is one of the most important open problems on Artin groups, dating back to Brieskorn, Arnol'd, Pham, and Thom in the '60s \cite{brieskorn1973groupes,van1983homotopy}.
Due to its numerous consequences and connections, over the years it attracted the attention of mathematicians from several areas. It was solved for certain classes of Artin groups with approaches from algebraic topology, geometric group theory, and combinatorics.

The $K(\pi, 1)$ conjecture states that a certain topological space $Y$, constructed from the geometric action of a Coxeter group $W$ on the Tits cone, is a classifying space (or ``$K(\pi, 1)$'') for the corresponding Artin group $G_W$.
If $W$ is a finite or affine reflection group acting on $\R^k$, then $Y$ is the complement in $\C^n$ of the complexification of all reflection hyperplanes (see \Cref{fig:S3}, left, and \Cref{fig:triangle-groups}, center).

In this paper, we focus on a combinatorial approach which has been very fruitful when studying the class of spherical Artin groups and, more recently, affine Artin groups.
The fundamental idea, which goes back to Garside \cite{garside1969braid}, is the following: given a group $G$ (in our case, an Artin group), fix a generating set and find a special element whose divisors generate the whole group and form a lattice under the divisibility relation.
For a spherical Artin group with its standard generating set, such a special element can be obtained by lifting the longest element of the corresponding finite Coxeter group $W$.
In this case, the lattice of divisors is isomorphic to $W$ with the weak Bruhat order.
When a group admits a special element, then it is called a \emph{Garside group}.
Together, the data of a Garside group, its generating set, and its special element form a \emph{Garside structure}.
Garside structures are very useful, thanks to an elegant solution to the word problem and an explicit combinatorial construction of a classifying space.
Deligne's proof of the $K(\pi, 1)$ conjecture for spherical Artin groups \cite{deligne1972immeubles} can be reinterpreted with the language of Garside structures (although Garside groups were introduced later \cite{dehornoy1999gaussian, dehornoy2002groupes, dehornoy2015foundations}).

The aforementioned Garside structure on spherical Artin groups is arguably one of the greatest milestones in the study of Artin groups.
However, it does not generalize beyond the spherical case, since infinite Coxeter groups have no longest element.
Birman-Ko-Lee \cite{birman1998new} and Bessis \cite{bessis2003dual} introduced and studied alternative ``dual'' realizations of spherical Artin groups as Garside groups, using a different generating set (a lift of all reflections in $W$) and different special elements (Coxeter elements).
Besides providing new interesting insights into the spherical case, this dual approach has the advantage that Coxeter elements exist also in infinite Coxeter groups.
For some families of affine Artin groups, Digne showed that this larger generating set together with a Coxeter element form a dual Garside structure \cite{digne2006presentations,digne2012garside}, but McCammond proved that the divisors of a Coxeter element do not form a lattice in almost all of the remaining affine Artin groups \cite{mccammond2015dual}.
On the positive side, McCammond and Sulway \cite{mccammond2017artin} exhibited a way to embed any affine Artin group into a Garside group, thus recovering some of the benefits of a Garside structure such as a solution to the word problem.
More recently, Salvetti and the author succeeded in proving the $K(\pi, 1)$ conjecture for all affine Artin groups \cite{paolini2021proof} with an approach that is based on the dual structure even though the lattice property does not necessarily hold.

The purpose of this paper is to outline this dual approach to the $K(\pi, 1)$ conjecture.
We are going to discuss the various combinatorial, topological, and geometric ingredients that went into the proof of the affine case given in  \cite{paolini2021proof}, how they might possibly generalize to other Artin groups, and the multitude of questions that naturally arise along the way.
We do not know to what extent the dual approach is viable to obtain a full proof of the $K(\pi, 1)$ conjecture. Regardless, we want to lay the foundations for exploring new exciting directions.

\smallskip

\noindent\textbf{Acknowledgments.}
The material of this paper is based on several talks that I gave between 2019 and 2021. In alphabetical order, I am extremely grateful to Emanuele Delucchi, Jon McCammond, and Mario Salvetti for our collaborations and the profound influence they had on my work.
Without any of them, this paper (and many others) would not exist.

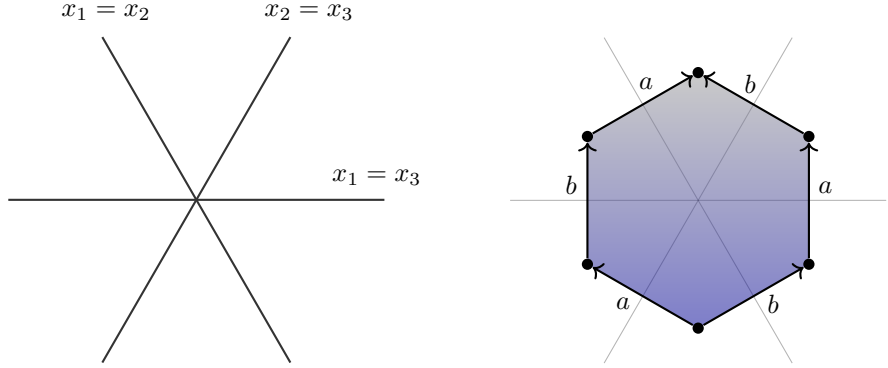
\begin{figure}[t]
    \begin{center}
	\begin{tikzpicture}[scale=1]
		\begin{scope}[thick, black!80]
		\draw (-2.5,0) -- (2.5,0);
		\draw[rotate=60] (-2.5,0) -- (2.5,0);
		\draw[rotate=120] (-2.5,0) -- (2.5,0);
		\end{scope}
		
		\node (A) at (2.4,0.3) {$x_1=x_3$};
		\node (B) at (1.5,2.5) {$x_2=x_3$};
		\node (C) at (-1.2,2.5) {$x_1=x_2$};
	\end{tikzpicture}%
	\qquad%
	\tikzstyle{salvettivertices}=[every node/.style={circle,inner sep=1.5pt,fill=black}]
	\tikzstyle{salvettiedges}=[every path/.style={->,thick}]
	\tikzstyle{XC}=[every path/.style={red, dashed, very thick}, every node/.style={red}]
	\begin{tikzpicture}[scale=1]
		\begin{scope}[every path/.style={opacity=0.3}]
		\draw (-2.5,0) -- (2.5,0);
		\draw[rotate=60] (-2.5,0) -- (2.5,0);
		\draw[rotate=120] (-2.5,0) -- (2.5,0);
		\end{scope}
		
		\begin{scope}
		\foreach \i in {0,1,...,5} {
			\node (C\i) at (\i*60+30:1.7) {};
		}
		\end{scope}

		\foreach \i in {0,1,...,5} {
			\coordinate (P\i) at (C\i) {};
		}
		\begin{scope}[every path/.style={opacity=0.3, left color=blue, right color=white}]
		\fill[shading angle=180] (P3) to (P2) to (P1) to (P0) to (P5) to (P4) to (P3);
		\end{scope}
		
		\begin{scope}[salvettivertices]
		\foreach \i in {0,1,...,5} {
			\node (D\i) at (\i*60+30:1.7) {};
		}
		\end{scope}
		
		\begin{scope}[salvettiedges]
		\draw (D3) to (D2);
		\draw (D2) to (D1);
		\draw (D0) to (D1);
		\draw (D4) to (D3);
		\draw (D4) to (D5);
		\draw (D5) to (D0);
		\end{scope}
		
		\node (a) at (6:1.7) {$a$};
		\node (b) at (66:1.7) {$b$};
		\node (c) at (114:1.7) {$a$};
		\node (d) at (173:1.7) {$b$};
		\node (e) at (234:1.7) {$a$};
		\node (f) at (306:1.7) {$b$};
	\end{tikzpicture}
	\end{center}
	\caption{The symmetric group $\mathfrak{S}_3$ on a three-element set is a spherical Coxeter group with presentation $\mathfrak{S}_3 = \< a,b \mid a^2 = b^2 = (ab)^3 = 1 \>$. It can be represented as the group of linear isometries of $\R^3$ that permutes the three coordinates. This representation can be made essential by restricting to the plane $\{x_1+x_2+x_3=0\} \subseteq \R^3$.
	\textbf{Left:} the reflection lines (corresponding to the transpositions of $\mathfrak{S}_3$) in the plane $\{x_1+x_2+x_3=0\}$.
	\textbf{Right:} the Salvetti complex $X_{\mathfrak{S}_3}$, which has one $0$-cell, two $1$-cells (labeled $a$ and $b$), and one hexagonal $2$-cell.
	The fundamental group of $X_{\mathfrak{S}_3}$ is the braid group $G_{\mathfrak{S}_3} = \< a, b \mid aba = bab \>$.
	}
	\label{fig:S3}
\end{figure}

\section{The general picture}
\label{sec:general-picture}

\subsection{Coxeter groups}
\label{sec:coxeter-groups}
A \emph{Coxeter group} is a group with a presentation of the following form:
\begin{equation}
	\label{eq:coxeter-presentation}
 	W = \< S \mid (st)^{m(s,t)} = 1 \;\; \forall\, s,t \in S \text{ such that $m(s,t) \neq \infty$} \>,
\end{equation}
where $S$ is a finite set and $m\colon S \times S \to \N$ is any symmetric integer matrix with $1$'s on the diagonal and all other entries $\geq 2$.
The off-diagonal entries are allowed to take the value $m(s, t) = \infty$, in which case there is no relation involving $s$ and $t$.
The $1$'s on the diagonal ensure that all generators have order $2$.
In this section, we recall some facts about Coxeter groups that are particularly relevant for the rest of this paper. We point to \cite{bourbaki1968elements, humphreys1992reflection} for a detailed introduction.

The presentation \eqref{eq:coxeter-presentation} can be encoded into a graph called the \emph{Coxeter graph}: the vertices are indexed by $S$, and there is an edge connecting $s$ and $t$ whenever $m(s,t) \geq 3$; this edge is labeled by $m(s,t)$ if $m(s,t) \geq 4$.
We always assume that $W$ is \emph{irreducible}, i.e., its Coxeter graph is connected.
The size of the generating set $S$ is called the \emph{rank} of $W$.
Any subset $T\subseteq S$ generates a subgroup of $W$ called a \emph{standard parabolic subgroup}, which is itself a Coxeter group, with a presentation obtained by restricting \eqref{eq:coxeter-presentation} to the generators in $T$ and the relations between them.
Any conjugate of $S$ can be used in place of $S$ to write a presentation for $W$ of the form \eqref{eq:coxeter-presentation}.
The conjugates of $S$ are called \emph{sets of simple reflections}.
Usually, a set $S$ of simple reflections is fixed, in which case the elements of $S$ are called \emph{simple reflections}. However, when taking the dual point of view (\Cref{sec:dual-approach}), it is useful to think of a Coxeter group without any preferred set of simple reflections.

For us, a Coxeter group $W$ always carries with it a fixed set $S$ of simple reflections (the pair $(W, S)$ is usually called a \emph{Coxeter system}) or the set of all sets of simple reflections (when the particular choice of $S$ is not relevant).
Indeed, two Coxeter groups may be isomorphic as groups (if we forget about simple reflections) but not as Coxeter groups.
For example, the dihedral group $W = \< s, t \mid s^2 = t^2 = (st)^6 = 1\>$ is irreducible when using $s,t$ as simple reflections, but it is the direct product of two Coxeter subgroups: the symmetric group $\mathfrak{S}_3$ (generated by $s$ and $tst$), and $\Z_2$ (generated by $(st)^3$). See for example \cite{muhlherr2006isomorphism}.

The family of Coxeter groups encompasses all discrete groups generated by (linear or affine) Euclidean reflections in $\R^k$ (we call them \emph{(real) reflection groups}).
The linear reflection groups are exactly the finite Coxeter groups, and they are also called \emph{spherical} because they act on the unit sphere $\S^{k-1} \subseteq \R^k$.
The action on $\S^{k-1}$ is cocompact, provided that $k=|S|$ (the representation is \emph{essential}).
For example, the symmetric groups are spherical Coxeter groups (\Cref{fig:S3}).
The infinite Euclidean reflection groups are called \emph{affine} Coxeter groups.
They act cocompactly on $\R^k$, provided that $k=|S|-1$.
See \Cref{fig:triangle-groups} for more examples.

\begin{figure}
    \makeatletter
    \define@key{z sphericalkeys}{radius}{\def\myradius{#1}}
    \define@key{z sphericalkeys}{theta}{\def\mytheta{#1}}
    \define@key{z sphericalkeys}{phi}{\def\myphi{#1}}
    \tikzdeclarecoordinatesystem{z spherical}{%
        \setkeys{z sphericalkeys}{#1}%
        \pgfpointxyz{\myradius*sin(\mytheta)*cos(\myphi)}{\myradius*sin(\mytheta)*sin(\myphi)}{\myradius*cos(\mytheta)}}
    \makeatother
    
    \begin{tikzpicture}[scale=2.5]
        \tdplotsetmaincoords{10}{20}
        \begin{scope}[tdplot_main_coords]

            \tdplotsetrotatedcoords{0}{90}{0}
            \begin{scope}[tdplot_rotated_coords]
                \newcommand{\beginphi}{96}
                \draw plot[variable=\x,domain=\beginphi:\beginphi+180,smooth,samples=60]  (z spherical cs:radius=1,theta=90,phi=\x);
            \end{scope}
            
            \tdplotsetrotatedcoords{90}{90}{0}
            \begin{scope}[tdplot_rotated_coords]
                \newcommand{\beginphi}{95}
                \draw plot[variable=\x,domain=\beginphi:\beginphi+180,smooth,samples=60]  (z spherical cs:radius=1,theta=90,phi=\x);
            \end{scope}
            
            \tdplotsetrotatedcoords{45}{54.7356103}{0}
            \begin{scope}[tdplot_rotated_coords]
                \newcommand{\beginphi}{100}
                \draw plot[variable=\x,domain=\beginphi:\beginphi+180,smooth,samples=60]  (z spherical cs:radius=1,theta=90,phi=\x);
            \end{scope}
            
            \tdplotsetrotatedcoords{-45}{54.7356103}{0}
            \begin{scope}[tdplot_rotated_coords]
                \newcommand{\beginphi}{100}
                \draw plot[variable=\x,domain=\beginphi:\beginphi+180,smooth,samples=60]  (z spherical cs:radius=1,theta=90,phi=\x);
            \end{scope}
            
            \tdplotsetrotatedcoords{45}{125.2643897}{0}
            \begin{scope}[tdplot_rotated_coords]
                \newcommand{\beginphi}{100}
                \draw plot[variable=\x,domain=\beginphi:\beginphi+180,smooth,samples=60]  (z spherical cs:radius=1,theta=90,phi=\x);
            \end{scope}
    
            \tdplotsetrotatedcoords{-45}{125.2643897}{0}
            \begin{scope}[tdplot_rotated_coords]
                \newcommand{\beginphi}{95}
                \draw plot[variable=\x,domain=\beginphi:\beginphi+180,smooth,samples=60]  (z spherical cs:radius=1,theta=90,phi=\x);
            \end{scope}
    
            \tdplotsetrotatedcoords{45}{90}{0}
            \begin{scope}[tdplot_rotated_coords]
                \newcommand{\beginphi}{95}
                \draw[dashed, orange] plot[variable=\x,domain=\beginphi:\beginphi+180,smooth,samples=60]  (z spherical cs:radius=1,theta=90,phi=\x);
            \end{scope}
        \end{scope}
        \draw[black!50, thick] (0,0,0) circle (1);
        
        \node at (0.35,0.11) {\small $a$};
        \node at (-0.2,-0.16) {\small $b$};
        \node at (0.37,-0.39) {\small $c$};
    \end{tikzpicture}
    \quad
    \begin{tikzpicture}[scale=1.5,
	extended line/.style={shorten >=-#1,shorten <=-#1}, extended line/.default=35cm]
	
	\clip (-1.23, -2.2) rectangle (2.09, 1.2);
	
	\begin{scope}[black!80]
	\newcommand{\rows}{5}
	\foreach \row in {-\rows, ...,\rows} {
		\draw [extended line] ($\row*({0.5*sqrt(3)}, 0.5)$) -- ($(0,\rows)+\row*({0.5*sqrt(3)}, -0.5)$);
		\draw [extended line] ($\row*(0, 1)$) -- ($({\rows/2*sqrt(3)}, \rows/2)+\row*({-0.5*sqrt(3)}, 0.5)$);
		\draw [extended line] ($\row*(0, 1)$) -- ($({-\rows/2*sqrt(3)}, \rows/2)+\row*({0.5*sqrt(3)}, 0.5)$);
	}
	\end{scope}
	
	\draw [extended line, dashed, orange] ($({0.25*sqrt(3)}, 0)$) -- ($({0.25*sqrt(3)}, 1)$);

        \node at (0.57,-0.2) {\small $a$};
        \node at (-0.1,-0.5) {\small $b$};
        \node at (0.57,-0.8) {\small $c$};
	\end{tikzpicture}
	\quad
	\input{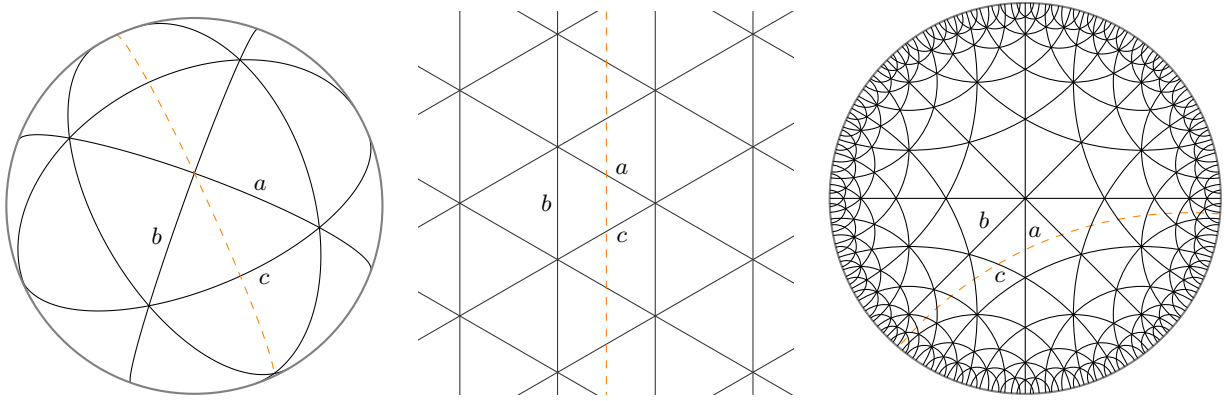}

    \caption{Reflection arrangements of some Coxeter groups of rank $3$ (also known as \emph{triangle groups}).
    The dashed line is the axis of the Coxeter element $w = abc$ (see \Cref{sec:coxeter-elements}).
    \textbf{Left:} the spherical $(2, 3, 3)$ triangle group, a.k.a.\ the symmetric group $\mathfrak{S}_4$.
    \textbf{Center:} the affine $(3, 3, 3)$ triangle group, a.k.a.\ the affine symmetric group of type $\tilde A_3$.
    \textbf{Right:} the hyperbolic $(4, 3, 3)$ triangle group.
    In all three cases, the triple $(p, q, r)$ consists of the upper-triangular entries of the Coxeter matrix.
    The sphere $\S^2$ (left), Euclidean plane $\R^2$ (center), and the hyperbolic plane $\H^2$ (right) are tiled by triangles with angles $\frac{\pi}{p}, \frac{\pi}{q}, \frac{\pi}{r}$.}
    \label{fig:triangle-groups}
\end{figure}

The representations of finite and affine Coxeter groups as Euclidean reflection groups inspired a more general \emph{geometric representation} for an arbitrary Coxeter group $W$ as a group generated by linear reflections with respect to a suitable bilinear form $B$ in $\R^n$ with $n=|S|$ \cite[Section 5.3]{humphreys1992reflection}.
The bilinear form is defined using the data of the presentation \eqref{eq:coxeter-presentation}. It is positive definite if $W$ is finite, positive semi-definite if $W$ is affine, and it otherwise admits both positive and negative vectors.

By switching to the contragradient representation \cite[Section 5.13]{humphreys1992reflection}, one gains a geometric picture that resembles the Euclidean case: $W$ is generated by reflections with respect to hyperplanes of $\R^n$, these hyperplanes divide $\R^n$ into simplicial cones called \emph{chambers}, and the generating set $S$ consists of the reflections with respect to the walls of a \emph{fundamental chamber} $C_0 \subseteq \R^n$.
The union of all $W$-translates $w(\bar{C}_0)$ forms a convex cone $I$ called the \emph{Tits cone}.
The closure $\bar C_0$ of the fundamental chamber is a fundamental domain for the action of $W$ on the Tits cone $I$.
In addition, $W$ acts simply transitively on the chambers contained in the Tits cone $I$, so the choice of a fundamental chamber induces a bijection between the elements of $W$ and the chambers in $I$.
The Tits cone is the whole space $\R^n$ if $W$ is finite, an open half-space (with the origin added) if $W$ is affine, and otherwise it does not contain any line.

The Coxeter groups that are neither finite nor affine are coarsely classified as:
\emph{Lorentzian}, if the bilinear form $B$ has exactly one negative eigenvalue; \emph{higher-rank},\footnote{Here \emph{rank} does not refer to the rank of the Coxeter group as defined earlier. Rather, it refers to the rank of the Lie group $SO(p, q)$ where $(p, q)$ is the signature of the bilinear form $B$.} if $B$ has at least two negative eigenvalues.
A Lorentzian Coxeter group is called \emph{hyperbolic} if (i) the bilinear form $B$ is non-degenerate and (ii) every vector in the interior of the Tits cone is negative.
Hyperbolic Coxeter groups act by isometries on the hyperbolic space $\H^{n-1}$ (realized as the hyperboloid model inside the Tits cone) and the chambers are simplicial (possibly with ideal vertices).
Higher-rank Coxeter groups act by isometries on the projectivization of the Tits cone with the Hilbert metric \cite{mcmullen2002coxeter}.

\subsection{The $K(\pi, 1)$ conjecture and Artin groups}
\label{sec:conjecture}
Let $\A$ be the set of all fixed hyperplanes of reflections of $W$, where $W$ acts on the Tits cone $I \subseteq\R^n$ via the representation introduced in the previous section.

\begin{conjecture}[$K(\pi, 1)$ conjecture]
    The space
    \[ Y = (I\times I) \setminus \bigcup_{H \in \A} (H \times H) \]
    is a $K(\pi, 1)$ space.
\end{conjecture}

This conjecture dates back to the '60s when it was proved for the symmetric group by Fox and Neuwirth \cite{fox1962braid}.
It was then proved for most finite Coxeter groups by Brieskorn \cite{brieskorn1973groupes} and for all finite Coxeter groups by Deligne \cite{deligne1972immeubles}.
The above formulation for general Coxeter groups is attributed to Arnol'd, Pham, and Thom \cite{van1983homotopy}.
If $W$ is finite or affine, then the $K(\pi, 1)$ conjecture says that the complement of the complexification of a (locally finite) Euclidean reflection arrangement is a $K(\pi, 1)$ space.
See \cite{paris2012k} for a survey on this problem.

The primary motivation for the $K(\pi, 1)$ conjecture, besides its elegance, comes from the study of \emph{Artin groups}.
To every Coxeter group $W$ presented as in \eqref{eq:coxeter-presentation}, there is an associated Artin group defined as follows:
\begin{equation}
    \label{eq:artin-presentation}
 	G_W = \< S \mid \!\!\underbrace{stst\dotsm}_{m(s,t) \text{ terms}} \! = \! \underbrace{tsts\dotsm}_{m(s,t) \text{ terms}} \forall\, s,t \in S \text{ such that $m(s,t) \neq \infty$} \>.
\end{equation}
The Artin group $G_W$ also arises as the fundamental group of the quotient space $Y_W = Y / W$ \cite{van1983homotopy, salvetti1994homotopy}.
The quotient $Y \to Y_W$ is a covering map, so the $K(\pi, 1)$ conjecture can be equivalently formulated by asking that $Y_W$ be a \emph{classifying space} for the Artin group $G_W$.
The space $Y_W$ is called the \emph{orbit configuration space} associated with $W$.
This name comes from the case of the symmetric group $W = \mathfrak{S}_n$, for which $Y_W$ is the space of configurations of $n$ (indistinguishable) points in $\R^2$ and $G_W$ is the \emph{braid group} on $n$ strands (see \Cref{fig:braids}).

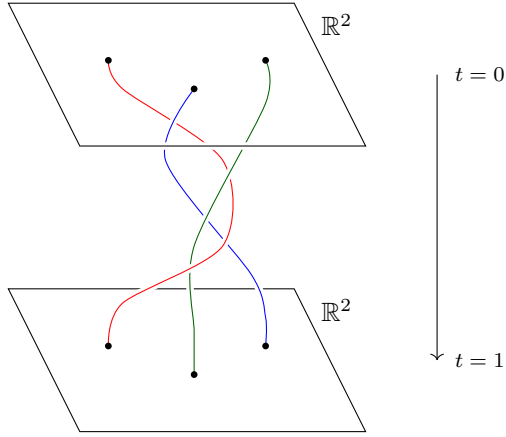
\begin{figure}
    \begin{center}
	\definecolor{darkgreen}{RGB}{0,100,0}
	\begin{tikzpicture}[scale=1.9]
	\draw (-1, -1.5) -- (1, -1.5) -- (1.5, -2.5) -- (-0.5, -2.5) -- cycle;
	
	\coordinate (A) at (-0.3, 0.1);
	\coordinate (B) at (0.3, -0.1);
	\coordinate (C) at (0.8, 0.1);
	
	\coordinate (A1) at ($(A) + (0,-2)$);
	\coordinate (B1) at ($(B) + (0,-2)$);
	\coordinate (C1) at ($(C) + (0,-2)$);

	\newcommand{\braida}{(A) ($(A) + (0.1, -0.2)$) ($(A) + (0.8, -0.7)$) ($(A) + (0.8, -1.3)$) ($(A) + (0.1, -1.7)$) (A1)}
	
	\begin{scope}
	\clip (-1, -1) rectangle (1, 0.1);
	\draw [white, line width=1mm] plot [smooth] coordinates {\braida};
	\draw [red] plot [smooth] coordinates {\braida};
	\end{scope}
	
	\newcommand{\braidb}{(B) ($(B) + (-0.2, -0.5)$) ($(B) + (0.4, -1.3)$) ($(C) + (0, -1.8)$) (C1)}
	\draw [white, line width=1mm] plot [smooth] coordinates {\braidb};
	\draw [blue] plot [smooth] coordinates {\braidb};
	
	\newcommand{\braidc}{(C) ($(C) + (0, -0.3)$) ($(C) + (-0.5, -1.3)$) ($(B) + (0, -1.7)$) (B1)}
	\draw [white, line width=1mm] plot [smooth] coordinates {\braidc};
	\draw [darkgreen] plot [smooth] coordinates {\braidc};
	
	\begin{scope}
	\clip (-1, -2) rectangle (1, -1);
	\draw [white, line width=1mm] plot [smooth] coordinates {\braida};
	\draw [red] plot [smooth] coordinates {\braida};
	\end{scope}

	\begin{scope}[every node/.style={fill, circle, inner sep=0.9}]
	\node at (A) {};
	\node at (B) {};
	\node at (C) {};
	
	\node at (A1) {};
	\node at (B1) {};
	\node at (C1) {};
	\end{scope}
	\draw[white, line width=1mm] (-1, 0.5) -- (1, 0.5) -- (1.5, -0.5) -- (-0.5, -0.5) -- cycle;
	\draw (-1, 0.5) -- (1, 0.5) -- (1.5, -0.5) -- (-0.5, -0.5) -- cycle;
	
	\node at (1.3, 0.35) {$\R^2$};
	\node at (1.3, -1.65) {$\R^2$};
	
	\draw [->] (2, 0) -- (2, -2);
	\node at (2.3, 0) {\footnotesize $t=0$};
	\node at (2.3, -2) {\footnotesize $t=1$};
	
	\end{tikzpicture}
	\end{center}
	\caption{A loop in the orbit configuration space $Y_W$ for $W = \mathfrak{S}_3$. In this case, $Y_W$ is the space of configurations of $3$ points in $\R^2$. Elements of the fundamental group $G_{\mathfrak{S}_3} = \pi_1(Y_W)$ are homotopy classes of loops, also known as \emph{braids}.}
	\label{fig:braids}
\end{figure}

As shown by Salvetti \cite{salvetti1987topology, salvetti1994homotopy}, the orbit configuration space $Y_W$ has the homotopy type of a CW complex $X_W$ with $k$-cells indexed by the finite standard parabolic subgroups of rank $k$.
The CW complex $X_W$ is known as the \emph{Salvetti complex} of $W$ (see \Cref{fig:S3}).
The presentation \eqref{eq:artin-presentation} can be read off the $2$-skeleton of the Salvetti complex, thus providing a simple proof that $G_W \cong \pi_1(X_W)$.
Since $X_W$ is finite-dimensional, the $K(\pi, 1)$ conjecture for $W$ implies that the Artin group $G_W$ is torsion-free (a property that is not known in general).
In addition, the $K(\pi, 1)$ conjecture makes it possible to compute the homology and cohomology of an Artin group $G_W$ using the configuration space $Y_W$ or, equivalently, the Salvetti complex $X_W$, as done in several works already \cite{arnold1970some, fuks1970cohomology, cohen1973cohomology, cohen1978braid, vainshtein1978cohomologies, goryunov1978cohomology, goryunov1982cohomology, salvetti1994homotopy, de1996cohomology, charney1996finite, salvetti1997artin, de1999arithmetic, de1999stability, de2000cohomology, de2001arithmetic, callegaro2004integral, callegaro2005cohomology, callegaro2006homology, cohen2007homology, callegaro2008cohomology, callegaro2008cohomology2, callegaro2010k, salvetti2013combinatorial, paolini2018weighted, paolini2019local}.

Of course, having a CW model for $Y_W$ can be also useful to prove the $K(\pi, 1)$ conjecture.
This is particularly true if we want to approach the conjecture with combinatorial techniques, as in this case, we would rather work with CW complexes (indexed by combinatorial objects associated with $W$) than with ``raw'' topological spaces.
This combinatorial spirit is at the heart of the approach we discuss in this paper.

To date, the $K(\pi, 1)$ conjeture has been proved in the following cases: spherical Artin groups \cite{deligne1972immeubles} (see \Cref{sec:standard-approach}); affine Artin groups \cite{paolini2021proof} (see \Cref{sec:dual-approach} and the rest of this paper); $2$-dimensional and FC-type Artin groups \cite{charney1995k} (the proof is based on finding a CAT($0$) metric on the \emph{Deligne complex}, see also \cite{charney2016problems}).
It was previously proved for some subclasses of these Artin groups, with different methods: braid groups \cite{fox1962braid} and spherical Artin groups of type $C_n$, $D_n$, $G_2$, $I_2(m)$ \cite{brieskorn1972artin};\footnote{I was pointed out by Georges Neaime (private communication) that Brieskorn's construction is incorrect in the case $F_4$.}
affine Artin groups of type $\tilde A_n$, $\tilde C_n$ \cite{okonek1979dask} and of type $\tilde B_n$ \cite{callegaro2010k}; Artin groups of large type \cite{hendriks1985hyperplane}.

\subsection{The ``standard'' approach}
\label{sec:standard-approach}

The (right) Cayley graph of $W$ with respect to the generating set $S$ is the Hasse diagram of a partial order on $W$ known as the (right) \emph{weak Bruhat order} $\leq_S$:
\[
    u \leq_S v \quad \text{if and only if } l_S(v) = l_S(u) + l_S(u^{-1} v),
\]
where $l_S(u)$ is the length of $u$ with respect to $S$.
Choosing ``left'' instead of ``right'' (and replacing $u^{-1}v$ with $vu^{-1}$ in the definition of $\leq_S$) does not have an impact, since the resulting partial order $\leq_S'$ is isomorphic to $\leq_S$ via the map $u \mapsto u^{-1}$.
Notice that the edges of the Cayley graph are labeled by elements of $S$, so $(W, \leq_S)$ is an edge-labeled poset.

If $W$ is finite, then it has a unique longest element $\delta$ (where the length is measured by $l_S$).
Geometrically, $\delta$ is the element that sends the fundamental chamber $C_0$ to its opposite.
It is useful to think of the poset $(W, \leq_S)$ as the interval between $1$ and $\delta$ in the Cayley graph of $W$: every element $u \in W$ lies on at least one geodesic from $1$ to $\delta$, and the relation $u \leq_S v$ holds if and only if there is a geodesic from $1$ to $\delta$ which passes through $u$ and $v$ (in this order).
To emphasize this interval structure, we denote the poset $(W, \leq_S)$ by $[1, \delta]_S$.
For the symmetric group $\mathfrak{S}_3$, this is shown in \Cref{fig:weak-order}.

\begin{figure}
    \centering
    \begin{tikzpicture}[every path/.style={-latex}, scale=1.8]
	   	\begin{scope}[every node/.style={red!75!black}]
		   	\node (D0) at (30:1.2) {$ba$};
		   	\node (D1) at (90:1.2) {$\delta$};
		   	\node (D2) at (150:1.2) {$ab$};
		   	\node (D3) at (210:1.2) {$a$};
		   	\node (D4) at (270:1.2) {$1$};
		   	\node (D5) at (330:1.2) {$b$};
	   	\end{scope}
	   	
	   	\begin{scope}
		   	\node (b1) at (-60:1.15) {{$b$}};
		   	\node (a1) at (-120:1.15) {{$a$}};
		   	\node (b2) at (180:1.15) {{$b$}};
		   	\node (a2) at (0:1.15) {{$a$}};
		   	\node (b3) at (60:1.15) {{$b$}};
		   	\node (a3) at (120:1.15) {{$a$}};
	   	\end{scope}

	    \draw (D0) -- (D1);
	    \draw (D2) -- (D1);
	    \draw (D3) -- (D2);
	    \draw (D4) -- (D3);
	    \draw (D4) -- (D5);
	    \draw (D5) -- (D0);
    \end{tikzpicture}
    \qquad\qquad
    \begin{tikzpicture}[scale=1.8]
	    \coordinate (delta) at (90:1.2);
	    \coordinate (ab) at (150:1.2);
	    \coordinate (a) at (210:1.2);
	    \coordinate (1) at (270:1.2);
	    
	    \begin{scope}
		    \node (b1) at (-60:1.15) {{$b$}};
		    \node (a1) at (-120:1.15) {{$a$}};
		    \node (b2) at (180:1.15) {{$b$}};
		    \node (a2) at (0:1.15) {{$a$}};
		    \node (b3) at (60:1.15) {{$b$}};
		    \node (a3) at (120:1.15) {{$a$}};
	    \end{scope}
	    
	    \begin{scope}[every node/.append style = {font=\small}]
		    \draw (1) -- (a);
		    \draw (a) -- (ab);
		    \draw (ab) -- (delta);
		    \draw (1) -- (delta) node [midway, fill=white, inner sep=2.5] {{$\delta$}};
		    \draw (1) -- (ab) node [midway, pos=0.4, above] {\;\;\;{$ab$}};
		    \draw[dashed] (a) -- (delta) node [midway, pos=0.6, below, black!50] {\;\;\;$ba$};
	    \end{scope}
	    
	    \coordinate (ba) at (30:1.2);
	    \coordinate (deltabis) at (90:1.2);
	    \coordinate (1bis) at (270:1.2);
	    \coordinate (b) at (330:1.2);
	    
	    \begin{scope}[every node/.append style = {font=\small}]
		    \draw (1bis) -- (b);
		    \draw (b) -- (ba);
		    \draw (ba) -- (deltabis);
		
		    \draw (1bis) -- (ba) node [midway, pos=0.4, above] {{$ba$}\;\;\;};
		    \draw[dashed] (b) -- (deltabis) node [midway, pos=0.6, below, black!50] {$ab$\;\;\;};
	    \end{scope}
	    
	  	\node at (1) {\phantom{$1$}};
    \end{tikzpicture}
    \caption{\textbf{Left:} the interval $[1, \delta]_S$ for the symmetric group $\mathfrak{S}_3$. The edges are labeled by the simple reflections $a$ and $b$. The longest element is given by $\delta = aba = bab$.
    \textbf{Right:} the labeled order complex of $[1, \delta]_S$.
    Its quotient $K$ is a classifying space for the braid group $G_{\mathfrak{S}_3}$ and consists of the following simplices: the $0$-simplex $[\,]$; the five $1$-simplices $[a]$, $[b]$, $[ab]$, $[ba]$, and $[\delta]$; the six $2$-simplices $[a|b]$, $[b|a]$, $[a|ba]$, $[ab|a]$, $[ba|b]$, and $[b|ab]$; the two $3$-simplices $[a|b|a]$ and $[b|a|b]$.}
    \label{fig:weak-order}
\end{figure}
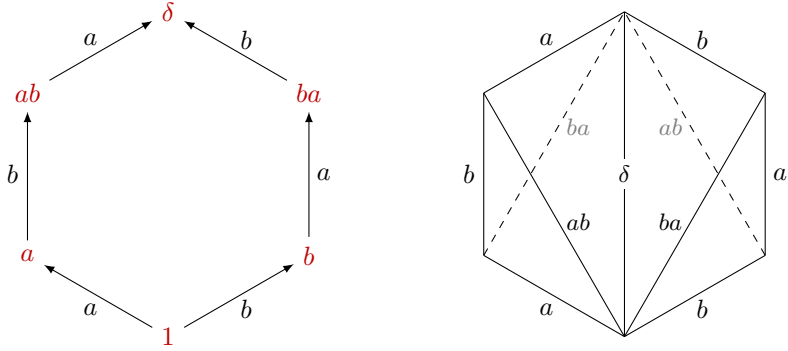

Crucially, the interval $[1, \delta]_S$ is a \emph{lattice}: every pair of elements has a unique minimal upper bound (a least common multiple) and a unique maximal lower bound (a greatest common divisor).
This was first shown by Deligne \cite{deligne1972immeubles}\footnote{Another proof of the lattice property was given later by Bj\"orner-Edelman-Ziegler \cite{bjorner1990hyperplane}. See also \cite[Section 3.2]{bjorner2006combinatorics}.} and used to prove the $K(\pi, 1)$ conjecture in the spherical case (actually, Deligne's proof works more generally for finite simplicial arrangements of linear hyperplanes in $\R^n$).
To prove the $K(\pi, 1)$ conjecture, Deligne showed that the universal cover of $Y_W$ is an increasing union of copies of the subspace determined by the \emph{positive paths}, which can be proved to be contractible using the lattice property.
It follows that the whole universal cover is itself contractible.
This idea goes back to Garside \cite{garside1969braid}.
See also \cite{delucchi2006topology, delucchi2009combinatorics, paris2012k, paolini2015thesis} for different reformulations of Deligne's proof.

One can also \emph{define} spherical Artin groups in terms of the interval $[1, \delta]_S$: $G_W$ is the group generated by the set $S$ and subject to all relations that identify any two words that can be read along maximal chains with the same initial and final point.
For example, in the symmetric group $\mathfrak{S}_3$ (\Cref{fig:weak-order}), the only relation is $aba = bab$ (obtained by reading the labels along the two geodesics from $1$ to $\delta$), so the Artin group is presented as $G_{\mathfrak{S}_3} = \<a, b \mid aba = bab\>$.
We say that $G_W$ is the \emph{interval group} associated with the labeled poset $[1, \delta]_S$.
This construction was generalized by Dehornoy and Paris, replacing $[1, \delta]_S$ with any labeled lattice $P$ satisfying some additional conditions.
They called \emph{Garside groups} the interval groups arising in this way \cite{dehornoy1999gaussian, dehornoy2003homology, charney2004bestvina, mccammond2005introduction}.
The case $P = [1, \delta]_S$ is known as the ``standard'' Garside structure on spherical Artin groups.

Elements in a Garside group admit a normal form $\delta^m y$, where $\delta$ is the top element of the defining lattice $P$, the exponent $m$ is an integer, and $y$ is a positive word in the generators.
In addition, a suitable quotient $K$ of the order complex of $P$ is a classifying space for the corresponding Garside group (see \Cref{fig:weak-order}).
More specifically, $K$ is obtained by identifying any two simplices $\{y_0 < y_1 < \dotsm < y_k\}$ and $\{z_0 < z_1 < \dotsm < z_k\}$ such that $y_i^{-1} y_{i+1} = z_i^{-1} z_{i+1}$ for all $i=0, \dotsc, k-1$.
Then a simplex of $K$ is uniquely determined by the sequence $(x_1, \dotsc, x_k) = (y_0^{-1} y_1, \dotsc, y_{k-1}^{-1} y_k)$ and we denote this simplex by $[x_1 | x_2 | \dotsb | x_k]$.
The complex $K$ was first introduced by Brady for braid groups \cite{brady2001partial} and then extended to spherical Artin groups by Brady-Watt \cite{brady2002k} and Bestvina \cite{bestvina1999non}, and to general Garside groups by Charney-Meier-Whittlesey \cite{charney2004bestvina}.
We call $K$ the \emph{interval complex} associated with $P$.
The lattice property of $P$ is crucial to obtain the normal form mentioned above and to prove that $K$ is a classifying space.

For spherical Artin groups $G_W$, one can show that $K$ is homotopy equivalent to the Salvetti complex $X_W$ and this is another way to prove the $K(\pi, 1)$ conjecture in the spherical case.
This homotopy equivalence is shown in \cite{delucchi2009combinatorics}.
Another complex that is homotopy equivalent to both $K$ and $X_W$ is the classifying space of the Artin monoid $G_W^+$, appearing in \cite{dobrinskaya2006configuration, ozornova2017discrete, paolini2017classifying}.

\subsection{The ``dual'' approach}
\label{sec:dual-approach}
The absence of the longest element in infinite Coxeter groups makes it impossible to extend the standard Garside structure to non-spherical Artin groups.
An alternative and promising direction to study Artin groups is based on a ``dual'' presentation of Coxeter groups, where the standard generating set $S$ is replaced by the set $R$ of all reflections (i.e., all conjugates of elements of $S$).
The associated length function $l_R$ is called the \emph{absolute length} and the induced partial order $\leq_R$ on $W$ is called the \emph{absolute order}.
Note that $R$ is infinite if $W$ itself is infinite, whereas $S$ is always finite.
The absolute order does not depend on the choice of left or right, thanks to the generating set $R$ being closed under conjugation.

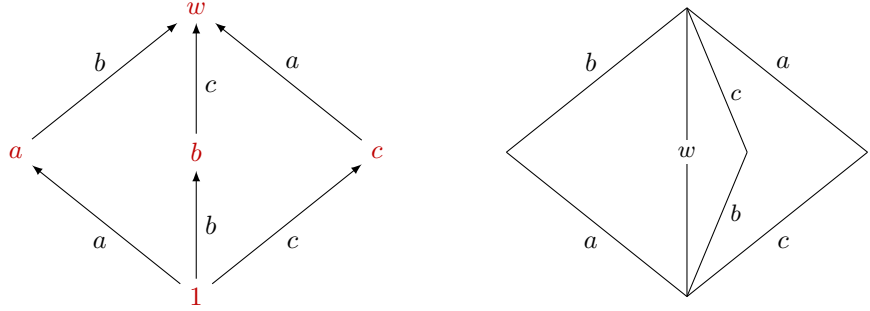
\begin{figure}
    \begin{center}
    \begin{tikzpicture}[every path/.style={-latex}, scale=1.6]
	    \begin{scope}[every node/.style={red!75!black}]
		    \node (P) at (0,0) {$1$};
		    \node (L1) at (-1.5,1.2) {$a$};
		    \node (L2) at (0,1.2) {$b$};
		    \node (L3) at (1.5,1.2) {$c$};
		    \node (R) at (0, 2.4) {$w$};
	    \end{scope}
	    
	    \node (a1) at (-0.8, 0.45) {{$a$}};
	    \node (b1) at (0.12, 0.6) {{$b$}};
	    \node (c1) at (0.8, 0.45) {{$c$}};
	    \node (b2) at (-0.8, 1.95) {{$b$}};
	    \node (c2) at (0.12, 1.75) {{$c$}};
	    \node (a2) at (0.8, 1.95) {{$a$}};
	    
	    \draw (P) -- (L1);
	    \draw (P) -- (L2);
	    \draw (P) -- (L3);
	    \draw (L1) -- (R);
	    \draw (L2) -- (R);
	    \draw (L3) -- (R);
    \end{tikzpicture}
    \qquad\qquad
    \begin{tikzpicture}[scale=1.6]
	    \coordinate (w) at (0, 2.4);
	    \coordinate (a) at (-1.5,1.2);
	    \coordinate (b) at (0.5,1.2);
	    \coordinate (c) at (1.5,1.2);
	    \coordinate (1) at (0,0);
	    
	    \begin{scope}[every node/.append style = {font=\small}]
		    \draw (a) -- (w);
		    \draw (1) -- (w) node [midway, pos=0.5, fill=white, inner sep=2.5] {{$w$}};
		    \draw (1) -- (a);
		    \draw (1) -- (c);
		    \draw (c) -- (w);
		    \draw (1) -- (b) node [midway, pos=0.7, below] {\;\;{$b$}};
		    \draw (b) -- (w) node [midway, pos=0.3, above] {\;\;{$c$}};
	    \end{scope}
	    
 	    \node (a1) at (-0.8, 0.45) {{$a$}};
	    \node (c1) at (0.8, 0.45) {{$c$}};
	    \node (b2) at (-0.8, 1.95) {{$b$}};
	    \node (a2) at (0.8, 1.95) {{$a$}};

	    \node at (1) {\phantom{$1$}};
    \end{tikzpicture}
    \end{center}
    \caption{\textbf{Left:} the ``dual'' interval $[1, w]_R$ for the symmetric group $\mathfrak{S}_3$ with $w = ab$ as the chosen Coxeter element. The edges are labeled by the three reflections $a,b,c$ (note that the third reflection $c$ is equal to the longest element $\delta = aba = bab$, but this is just a coincidence).
    \textbf{Right:} the labeled order complex of $[1, w]_R$.
    Its quotient $K$ is a classifying space for the (dual) braid group $G_{\mathfrak{S}_3}$ and consists of the following simplices: the $0$-simplex $[\,]$; the four $1$-simplices $[a]$, $[b]$, $[c]$, and $[w]$; the three $2$-simplices $[a|b]$, $[b|c]$, and $[c|a]$.}
    \label{fig:ncp}
\end{figure}

Let us restrict for now to finite Coxeter groups, where the absolute length of all $\leq_R$-maximal elements is equal to the rank of $W$.\footnote{This is not the case for infinite Coxeter groups, see for example \cite{lewis2019computing}.}
Among the maximal elements, a special role is played by \emph{Coxeter elements}, defined as $w = s_1s_2\dotsm s_n$ where $\{s_1, s_2, \dotsc, s_n\}$ is any set of simple reflections and the product is taken in any order.
If $W$ is finite, then all Coxeter elements $w\in W$ are conjugate and the intervals $[1,w]_R$ (inside the Cayley graph of $W$ with respect to the generating set $R$) are lattices (see \Cref{fig:ncp}).
This gives rise to several ``dual'' Garside structures, all isomorphic to each other.
Perhaps surprisingly, the interval group associated with any of these intervals $[1,w]_R$ is naturally isomorphic to the Artin group $G_W$.
Therefore, it is really the usual Artin group that we are studying and not some new Garside group.\footnote{This is not necessarily the case if $w$ is not a Coxeter element, see \cite{baumeister2021interval}.}

The duality between standard and dual presentations manifests itself in multiple numerical ``coincidences'', the most apparent being that $l_S(\delta) = |R|$ and $l_R(w) = |S|$.
In other words, the length of the standard interval $[1,\delta]_S$ is equal to the cardinality of $R$ (which is the set of atoms of $[1, w]_R$) and conversely the length of any dual interval $[1, w]_R$ is equal to the cardinality of $S$ (which is the set of atoms of $[1, \delta]_S$).
The dual presentation was first introduced by Birman-Ko-Lee \cite{birman1998new} for the braid group and then by Bessis \cite{bessis2003dual} for all finite Coxeter groups.

The intervals $[1,w]_R$ also exist in infinite Coxeter groups.
So it is natural to ask: Are they lattices, thus giving rise to Garside structures?\footnote{When the defining interval is infinite, the term \emph{quasi-Garside} is often used in place of \emph{Garside} (see for example \cite{dehornoy2015foundations}). In the present paper, we will not make this distinction.} Are the corresponding interval groups (called \emph{dual Artin groups}) isomorphic to the usual Artin groups? Can they help us find a solution to the word problem and the $K(\pi, 1)$ conjecture?

These questions are motivated by the success in understanding affine Artin groups by means of the dual approach. McCammond showed that $[1, w]_R$ fails to be a lattice in most affine cases \cite{mccammond2015dual}. However, McCammond and Sulway were able to prove that affine dual Artin groups are always isomorphic to the corresponding standard Artin groups, and can be included in larger Garside groups \cite{mccammond2017artin}.
In particular, this solves the word problem and shows that affine Artin groups are torsion-free.
More recently, Salvetti and the author proved the $K(\pi, 1)$ conjecture for affine Artin groups \cite{paolini2021proof}.
At a very high level, the proof consists of the following three components.
\begin{enumerate}
    \item Show that the interval complex $K_W$ associated with $[1,w]_R$ is a classifying space, despite the failure of the lattice property.
    \item Introduce a new subcomplex $X_W' \subseteq K_W$ with the same homotopy type as the orbit configuration space $Y_W$.
    \item Find a deformation retraction of $K_W$ onto $X_W'$.
\end{enumerate}
Together, these three steps imply that $Y_W$ is a classifying space and that the dual Artin group $\pi_1(K_W)$ is isomorphic to the standard Artin group $G_W = \pi_1(Y_W)$.
Therefore, they prove the $K(\pi, 1)$ conjecture and re-prove the isomorphism between standard and dual affine Artin groups.
In the rest of the present paper, we dive deeper into the different geometric, combinatorial, and topological aspects of this proof, with the hope that some of the key ideas can be generalized beyond the affine case.

\section{Coxeter elements}
\label{sec:coxeter-elements}

In this section, we introduce the main characters of the dual approach outlined in \Cref{sec:dual-approach}: Coxeter elements.
Let $W$ be a Coxeter group.
For any set $S = \{s_1, s_2, \dotsc, s_n\}$ of simple reflections (not necessarily the one used to define $W$), we say that the product $w = s_1s_2\dotsm s_n$ is a \emph{Coxeter element} of $W$ \cite{coxeter1934discrete, coxeter1951product, humphreys1992reflection}.
Any order of the simple reflections $s_1, \dotsc, s_n$ can be used and different orders can give rise to different Coxeter elements.
It is noted in \cite[Lemma 3.8]{ingalls2009noncrossing} and \cite[Lemma 5.1]{paolini2021proof} that the reflection length of any Coxeter element is equal to $n = |S|$.
In other words, it is not possible to write a Coxeter element as a product of less than $n$ reflections.

If $W$ is finite, then all Coxeter elements form a single conjugacy class and the order of any Coxeter element is $h = 2|R|/n$ ($h$ is called the \emph{Coxeter number} of $W$).
The eigenvalues of $w$ are of the form $\zeta^e$, where $\zeta = e^{2\pi i/h}$ and $e$ runs through the exponents of $W$ (listed in \Cref{table:degrees}).
$\R^n$ is the orthogonal sum of $w$-invariant subspaces $V_1, \dotsc, V_k$ of dimension $1$ or $2$: for each exponent $e < h/2$, there is a plane $V_i$ where $w$ acts as a rotation of $2\pi e/h$; for each exponent $e = h/2$ (there are at most $2$ of them), there is a line $V_i$ where $w$ acts as a reflection with respect to the origin.
There is a unique $w$-invariant plane $P$ corresponding to the exponent $e=1$, called the \emph{Coxeter plane}.
The Coxeter element $w$ acts on $P$ as a rotation of $2\pi/h$.
The Coxeter plane is also interesting thanks to the following simple observation.

\begin{lemma}
    Let $w$ be a Coxeter element in an irreducible finite Coxeter group $W$ acting on the unit sphere $\S^{n-1} \subseteq \R^n$ with $n = |S|$.
    Let $P \subseteq \R^n$ be the Coxeter plane.
    Then the circle $P \cap \S^{n-1}$ is the set of points $x \in \S^{n-1}$ for which the spherical distance $d(x, w(x))$ is minimized.
    \label{lemma:spherical-axis}
\end{lemma}

\begin{proof}
    Given $x \in \S^{n-1}$, write $x = v_1 + \dotsb + v_k$ with $v_i \in V_i$.
    Then
    \begin{align*}
        \cos \big(d(x, w(x)) \big) &= \< x, w(x) \> \\
        &= \< v_1, w(v_1) \> + \dotsb + \< v_k, w(v_k) \> \\
        &= \| v_1 \|^2 \cos\theta_1 + \dotsb + \| v_k \|^2 \cos\theta_k,
    \end{align*}
    where $\theta_i$ is equal to the rotation angle of $w|_{V_i}$ if $V_i$ is a plane and to $\pi$ if $V_i$ is a line.
    Since $\|v_1\|^2 + \dotsb + \|v_k\|^2 = \|x \|^2 = 1$, the previous expression is maximized if and only if all $v_i$ are $0$ except for the one corresponding to the largest $\cos \theta_i$.
    Such an $i$ is the one for which $\theta_i = 2\pi/h$ and $V_i = P$ is the Coxeter plane.
\end{proof}

\begin{table}
    \begin{center}
    \renewcommand{\arraystretch}{1.15}
    \begin{tabular}{|c|l|c|c|}
        \hline
        Type & Exponents & $h$ & $|W|$ \\
        \hline
        $A_n$ & $1,2,3,\dotsc, n$ & $n+1$ & $(n+1)!$ \\\hline
        $B_n$ & $1,3,5\dotsc, 2n-1$ & $2n$ & $2^n n!$ \\\hline
        $D_n$ & $1,3,5\dotsc, 2n-3, n-1$ & $2n-2$ & $2^{n-1} n!$ \\\hline
        $E_6$ & $1,4,5,7,8,11$ & $12$ & $2^7 3^4 5$ \\\hline
        $E_7$ & $1,5,7,9,11,13,17$ & $18$ & $2^{10} 3^4 5\ 7$ \\\hline
        $E_8$ & $1,7,11,13,17,19,23,29$ & $30$ & $2^{14} 3^5 5^2 7$\\\hline
        $F_4$ & $1,5,7,11$ & $12$ & $48$ \\\hline
        $G_2$ & $1,5$ & $6$ & $12$ \\\hline
        $H_3$ & $1,5,9$ & $10$ & $120$ \\\hline
        $H_4$ & $1,11,19,29$ & $30$ & $14400$ \\\hline
        $I_2(m)$ & $1,m-1$ & $m$ & $2m$ \\\hline
    \end{tabular}
    \bigskip
    \end{center}
    \caption{Exponents, Coxeter number, and cardinality of irreducible finite Coxeter groups (see \cite[Sections 2.11, 2.13, and 3.7]{humphreys1992reflection} and \cite[Appendix A1]{bjorner2006combinatorics}).}
    \label{table:degrees}
\end{table}

The fact that all Coxeter elements form a single conjugacy class holds more generally whenever the Coxeter graph is a tree.
This applies to all finite Coxeter groups (discussed above), but also to several infinite Coxeter groups, including all irreducible affine Coxeter groups except for the infinite family $\tilde A_n$.
In an arbitrary Coxeter group, however, Coxeter elements can form more than one conjugacy class and exhibit substantially different geometric properties and different spectra.

In affine Coxeter groups, Coxeter elements act as hyperbolic isometries on the Euclidean space \cite{mccammond2015dual}.
The set of points $x \in \R^n$ minimizing the Euclidean distance $d(x, w(x))$ is a line called the \emph{Coxeter axis}.
Analogously, in hyperbolic Coxeter groups, Coxeter elements act as hyperbolic isometries on the hyperbolic space $\H^n$ by \cite[Lemma 5.5]{mccammond2021factoring} and thus possess an axis (a hyperbolic line consisting of the points that are minimally moved).
In view of \Cref{lemma:spherical-axis}, the circle $P \cap \S^{n-1}$ can be regarded as the axis of the Coxeter element $w$ in a finite Coxeter group.
Therefore, in all three geometries (spherical, Euclidean, and hyperbolic), Coxeter elements $w$ have a well-defined axis, and the axis is a $w$-invariant geodesic.
Coxeter axes in triangle groups are shown in \Cref{fig:triangle-groups}.

\begin{question}
    Do all Coxeter elements (in arbitrary Coxeter groups) have an axis? Is the axis not contained in any reflection hyperplane?
    \label{question:axis}
\end{question}

If the previous question has a positive answer, then the axis $\ell$ of a Coxeter element $w$ passes through the interior of several chambers that we call \emph{axial chambers} (as in \cite{mccammond2015dual}).
The following question was positively answered for affine Coxeter groups in \cite[Theorem 8.10]{mccammond2015dual} and \cite[Theorem 3.8]{paolini2021proof}.

\begin{question}
    Is it true that a Coxeter element can be written as the product of the reflections with respect to the walls of any axial chamber (in some order)?
    \label{question:axial-factorization}
\end{question}

To answer this question, it is enough to show that a Coxeter element has length $\leq n$ with respect to the generating set $S$ consisting of the reflections with respect to some axial chamber.
Indeed, the length then needs to be exactly $n$ (because the reflection length is $n$) and a factorization into reflections in $S$ needs to use all of them because Coxeter elements are \emph{essential} (not contained in any proper parabolic subgroup) \cite{paris2007irreducible}.
McMullen's interpretation of Coxeter axes as billiard trajectories \cite{mcmullen2002coxeter} could help answering \Cref{question:axial-factorization}.

When the Coxeter graph is bipartite, one can construct so-called \emph{bipartite Coxeter elements}: these are obtained as $w = s_1s_2 \dotsm s_n$ where $S = \{s_1, \dotsc, s_k \} \cup \{s_{k+1}, \dotsc, s_n\}$ is a bipartition of any set $S$ of simple reflections.
In other words, $s_i$ and $s_j$ commute whenever $i, j \leq k$ or $i, j \geq k+1$.
If the Coxeter graph is a tree, then every Coxeter element is a bipartite Coxeter element with respect to some set of simple reflections.
This is especially useful to study Coxeter elements in finite  \cite{humphreys1992reflection, bessis2003dual} and affine \cite{mccammond2015dual} Coxeter groups.
In fact, \cite[Section 3]{paolini2021proof} shows a clear dichotomy between bipartite and non-bipartite affine Coxeter elements, the former having a significantly simpler geometrical behavior.
Bipartite Coxeter elements also minimize the spectral radius among all Coxeter elements in a fixed (hyperbolic or higher-rank) Coxeter group \cite{mcmullen2002coxeter}.
In \Cref{fig:triangle-groups}, the Coxeter element on the left is bipartite, whereas the other two are not.

\section{Factoring Coxeter elements: the noncrossing partition poset $[1,w]$}
\label{sec:ncp}

As explained in \Cref{sec:dual-approach}, the dual approach is based on understanding the poset $[1, w] = [1, w]_R$ which encodes the combinatorial data of all minimal factorizations of a Coxeter element $w$ into reflections.
For this, it can be useful to understand the minimal factorizations of $w$ into arbitrary reflections of the ambient bilinear form, not necessarily belonging to $W$.
These factorization posets are studied in \cite{brady2002partial} for spherical isometries, in \cite{brady2015factoring} for Euclidean isometries, and in \cite{mccammond2021factoring} for arbitrary non-degenerate quadratic spaces.
In this general setting, a recurring theme is that (under certain hypotheses) an isometry $u$ below $w$ is uniquely determined by $w$ and by its moved space $\Mov(u) := \im(u - \id)$.
In the spherical case, intervals are easy to describe: for any isometry $w$ in the orthogonal group $O(n)$, the interval $[1, w]$ in $O(n)$ is isomorphic to the poset of all subspaces of $\Mov(w)$ ordered by inclusion.

Let us go back to the setting of Coxeter groups.
If $w \in W$ is a Coxeter element, the interval $[1, w] = [1,w]_R$ is called a (generalized) \emph{noncrossing partition poset}.
The terminology comes from the case where $W$ is the symmetric group $\mathfrak{S}_{n+1}$, whose reflections are all transpositions $(i \ j)$.
Here Coxeter elements are the $(n+1)$-cycles, such as $(1 \ 2 \ \cdots \ n+1)$.
Then, the poset $[1, w]$ is naturally isomorphic to the classical lattice of noncrossing partitions of an $(n+1)$-gon (see for instance \cite{armstrong2009generalized}).

The maximal chains in $[1, w]$ correspond to the minimal factorizations of $w$ as a product of reflections, $w = r_1 r_2 \dotsm r_n$.
For example, in $\mathfrak{S}_3$ we have three reflections: $a=(1 \ 2)$, $b = (2 \ 3)$, and $c = (1 \ 3)$. The Coxeter element $w = ab = (1 \ 2 \ 3)$ has three minimal factorizations: $w = ab = bc = ca$.
For this case, the interval $[1, w]$ is depicted in \Cref{fig:ncp}.
There is a natural action of the braid group on the set of all minimal factorizations of a Coxeter element. This is called the \emph{Hurwitz action} and is defined as follows: the $i$-th generator $\sigma_i$ of the braid group sends the factorization $r_1r_2 \dotsm r_n$ to the factorization $r_1r_2\dotsm r_{i-1}r_{i+1} \tilde r_i r_{i+2} \dotsm r_n$ where $\tilde r_i = r_{i+1}r_ir_{i+1}$.
In words, $\sigma_i$ swaps the reflections $r_i$ and $r_{i+1}$ while conjugating $r_i$ by $r_{i+1}$.
Similarly, $\sigma_i^{-1}$ swaps $r_i$ and $r_{i+1}$ while conjugating $r_{i+1}$ by $r_i$.
It is known that the Hurwitz action is transitive on the set of all minimal factorizations of $w$ \cite{bessis2003dual,igusa2010exceptional,baumeister2014note}.

In the spherical case, the interval $[1,w]$ in $W$ (let us temporarily denote it by $[1,w]^W$) is an induced subposet of the interval $[1,w]$ in the whole orthogonal group $O(n)$.
However, \cite[Example 3.31]{paolini2021proof} shows that this is not true in the affine case, where there can be elements $u,v \in [1,w]^W$ such that $u \leq v$ in $[1,w]^L$ but $u \not\leq v$ in $[1,w]^W$. Here $[1,w]^{\Isom(\R^n)}$ is the interval inside the group $\Isom(\R^n)$ of all Euclidean isometries of $\R^n$.

If $W$ is finite, then Bessis proved that every element $u \in [1, w]$ is a Coxeter element for the parabolic subgroup generated by the reflections $\leq u$ \cite[Lemma 1.4.3 and Proposition 1.6.1]{bessis2003dual}.
This result was extended to \emph{crystallographic} Coxeter groups\footnote{In this context, $W$ is crystallographic if it is the Weyl group of a symmetrisable Kac-Moody Lie algebra. This happens if and only if the following two conditions are satisfied: (1) $m(s,t) \in \{ 2,3,4,6,\infty \}$ for all $s \neq t$; (2) in each circuit of the Coxeter graph not containing the edge label $\infty$, the number of edges labelled $4$ (resp.\ $6$) is even \cite[Theorem B.2]{hubery2016categorification}.} by Hubery-Krause \cite[Corollary 5.8]{hubery2016categorification}.
In the affine case, it was also proved in \cite[Theorem 3.22]{paolini2021proof}.
Note that, if $W$ is infinite, the subgroup generated by the reflections below $u$ is not necessarily a parabolic subgroup.

\begin{question}
    Is any element $u \in [1, w]$ a Coxeter element for the subgroup of $W$ generated by the reflections $\leq u$?
    \label{question:coxeter-element-subgroup}
\end{question}

The previous question is closely related to the following ones.

\begin{question}
    Is the Hurwitz action transitive on the minimal reflection factorizations of any element $u \in [1,w]$?
    \label{question:transitivity}
\end{question}

\begin{question}
    Let $u = r_1r_2\dotsm r_k$ be a minimal reflection factorization of an element $u \in [1,w]$.
    Does the subgroup $\<r_1, r_2, \dotsc, r_k\>$ only depend on $u$ (and not on the chosen factorization)?
    \label{question:factorization-subgroup}
\end{question}

\begin{lemma}
    The following implications hold: \Cref{question:coxeter-element-subgroup} $\Rightarrow$ \Cref{question:transitivity} $\Rightarrow$ \Cref{question:factorization-subgroup}.
\end{lemma}

\begin{proof}
    If \Cref{question:coxeter-element-subgroup} has a positive answer, then \Cref{question:transitivity} also does because of the transitivity of the Hurwitz action on Coxeter elements.
    After applying a Hurwitz move to a minimal reflection factorization $u = r_1r_2 \dotsm r_k$, the subgroup $\<r_1, r_2, \dotsc, r_k\>$ does not change.
    Therefore \Cref{question:transitivity} implies \Cref{question:factorization-subgroup}.
\end{proof}

\section{Combinatorics of $[1,w]$}
\label{sec:combinatorics}

Both the standard and the dual structure are particularly powerful to study spherical Artin groups, because the corresponding intervals $[1, \delta]_S$ and $[1, w] = [1, w]_R$ are lattices (and therefore give rise to Garside structures).
In the spherical case, $[1,w]$ was shown to be a lattice by Bessis with a case-by-case proof \cite{bessis2003dual} and then later by Brady and Watt with a case-free proof \cite{brady2008non}.
Digne showed that $[1, w]$ is a lattice in the affine cases $\tilde A_n$ (for certain choices of the Coxeter element $w$) and $\tilde C_n$ \cite{digne2006presentations, digne2012garside}.
It turns out that these and $\tilde G_2$ are the only affine cases where the lattice property holds, as proved by McCammond \cite{mccammond2015dual}.
It seems reasonable to expect that most noncrossing partition posets are not lattices, but a general characterization is not known.

\begin{question}
    For which Coxeter groups $W$ (and Coxeter elements $w$) the noncrossing partition poset $[1,w]$ is a lattice?
    \label{question:lattice}
\end{question}

There is another combinatorial property of $[1, w]$ which emerged as part of the proof of the $K(\pi, 1)$ conjecture in the affine case: lexicographic shellability \cite{bjorner1983lexicographically}.
Shellability is ubiquitous in the theories of Coxeter groups and of subspace arrangements, where several naturally arising posets turn out to be shellable \cite{bjorner1980shellable,dyer1993hecke,gottlieb1998cohomology,athanasiadis2007shellability,davidson2014lexicographic,delucchi2017shellability,paolini2020shellability,pagaria2021representations}.

EL-shellability of $[1, w]$ was an essential ingredient in \cite{paolini2021proof} towards showing that the interval complex $K_W$ deformation retracts onto the subcomplex $X_W'$ which, in turn, is homotopy equivalent to the Salvetti complex $X_W$.
Technically, the proof did not just use the existence of \emph{any} EL-labeling of $[1,w]$.
Rather, it used the fact that a certain family of total orderings $\prec$ of the reflections $R_0 = R \cap [1, w]$ makes the natural labeling of $[1,w]$ an EL-labeling.
In this setting, the EL-labeling property can be phrased as follows: for every $u \in [1, w]$, there is exactly one minimal reflection factorization $u = r_1r_2 \dotsm r_k$ that is $\prec$-increasing (i.e., $r_1 \prec r_2 \prec \dotsb \prec r_k$); furthermore, this factorization is lexicographically smallest among all factorizations of $[1,u]$.
Shellability of $[1,w]$ for finite Coxeter groups was proved by Athanasiadis-Brady-Watt \cite{athanasiadis2007shellability}.

The idea for the construction of suitable orderings $\prec$ of $R_0$ is geometric.
Suppose that Questions \ref{question:axis} and \ref{question:axial-factorization} have a positive answer. In particular, the Coxeter element $w$ has an axis $\ell$.
Fix an axial chamber $C_0$ and a point $p \in C_0 \cap \ell$.
When working in the contragradient representation of $W$, the axis becomes a two-dimensional plane $P$ (for finite $W$, this is the Coxeter plane), and the point $p$ becomes a line $\bar p \subseteq P$ through the origin.
Any reflection hyperplane $H$ intersects $P$ in a line through the origin (\Cref{question:axis} asks that no reflection hyperplane $H$ contains $P$).
The lines through the origin in $P$ have a natural cyclic ordering $\prec_c$ based on the orientation of the axis $\ell$.
The line $\bar p$ provides a way to make this cyclic ordering into a total ordering: given two lines through the origin $\lambda, \lambda'$, we say that $\lambda \prec \lambda'$ if $\bar p, \lambda, \lambda'$ appear in this order in $\prec_c$.
Then we can order reflections $r \in R_0$ based on the intersection between the reflection hyperplane $H$ and the Coxeter plane $P$.
See \Cref{fig:spherical-EL-labeling}.

\begin{figure}
    \begin{center}
	\makeatletter
	\define@key{z sphericalkeys}{radius}{\def\myradius{#1}}
	\define@key{z sphericalkeys}{theta}{\def\mytheta{#1}}
	\define@key{z sphericalkeys}{phi}{\def\myphi{#1}}
	\tikzdeclarecoordinatesystem{z spherical}{%
		\setkeys{z sphericalkeys}{#1}%
		\pgfpointxyz{\myradius*sin(\mytheta)*cos(\myphi)}{\myradius*sin(\mytheta)*sin(\myphi)}{\myradius*cos(\mytheta)}}
	\makeatother
	
	\begin{tikzpicture}[scale=2.5]
	\tdplotsetmaincoords{10}{20}
	\begin{scope}[tdplot_main_coords]

	\tdplotsetrotatedcoords{0}{90}{0}
	\begin{scope}[tdplot_rotated_coords]
	\newcommand{\beginphi}{96}
	\draw plot[variable=\x,domain=\beginphi:\beginphi+180,smooth,samples=60]  (z spherical cs:radius=1,theta=90,phi=\x);
	\end{scope}
	
	\tdplotsetrotatedcoords{90}{90}{0}
	\begin{scope}[tdplot_rotated_coords]
	\newcommand{\beginphi}{95}
	\draw plot[variable=\x,domain=\beginphi:\beginphi+180,smooth,samples=60]  (z spherical cs:radius=1,theta=90,phi=\x);
	\end{scope}
	
	\tdplotsetrotatedcoords{45}{54.7356103}{0}
	\begin{scope}[tdplot_rotated_coords]
	\newcommand{\beginphi}{100}
	\draw plot[variable=\x,domain=\beginphi:\beginphi+180,smooth,samples=60]  (z spherical cs:radius=1,theta=90,phi=\x);
	\end{scope}
	
	\tdplotsetrotatedcoords{-45}{54.7356103}{0}
	\begin{scope}[tdplot_rotated_coords]
	\newcommand{\beginphi}{100}
	\draw plot[variable=\x,domain=\beginphi:\beginphi+180,smooth,samples=60]  (z spherical cs:radius=1,theta=90,phi=\x);
	\end{scope}
	
	\tdplotsetrotatedcoords{45}{125.2643897}{0}
	\begin{scope}[tdplot_rotated_coords]
	\newcommand{\beginphi}{100}
	\draw plot[variable=\x,domain=\beginphi:\beginphi+180,smooth,samples=60]  (z spherical cs:radius=1,theta=90,phi=\x);
	\end{scope}
	
	\tdplotsetrotatedcoords{-45}{125.2643897}{0}
	\begin{scope}[tdplot_rotated_coords]
	\newcommand{\beginphi}{95}
	\draw plot[variable=\x,domain=\beginphi:\beginphi+180,smooth,samples=60]  (z spherical cs:radius=1,theta=90,phi=\x);
	\end{scope}
	
	\tdplotsetrotatedcoords{45}{90}{0}
	\begin{scope}[tdplot_rotated_coords]
	\newcommand{\beginphi}{95}
	\draw[dashed, orange] plot[variable=\x,domain=\beginphi:\beginphi+180,smooth,samples=60]  (z spherical cs:radius=1,theta=90,phi=\x);
	\end{scope}
	\end{scope}
	\draw[black!50, thick] (0,0,0) circle (1);
	
	\node at (0.35,0.11) {\small $a$};
	\node at (-0.2,-0.16) {\small $b$};
	\node at (0.37,-0.39) {\small $c$};
	\node at (-0.1,0.8) {\small $d$};
	\node at (0.62,0.35) {\small $e$};
	\node at (-0.6,-0.15) {\small $f$};
	
	\node[fill=black, circle, inner sep=0.03cm, label={[label distance=-0.05cm]0:$p$}] at (0.148, -0.15) {};
	\end{tikzpicture}
    \end{center}
	\caption{Procedure to define an axial ordering of the reflections in $\mathfrak{S}_4$. The Coxeter element is $w = abc$ (as in \Cref{fig:triangle-groups}) and its axis is the dashed line, oriented upwards. Starting from the point $p$ and moving along the axis, we encounter the reflections in the following order: $a \prec b \prec d \prec e \prec f \prec c$. Note that it is possible to swap $a$ and $b$ (they commute and fix the same point of the axis), as well as $e$ and $f$.
    }
	\label{fig:spherical-EL-labeling}
\end{figure}
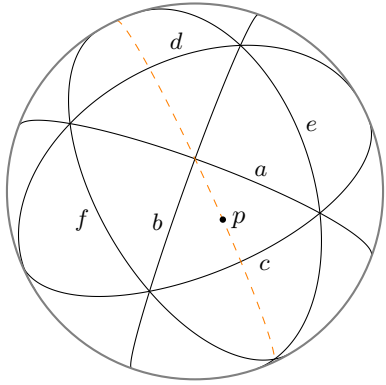

The procedure we have described is ambiguous whenever there are two reflection hyperplanes that intersect $P$ in the same line.
For finite Coxeter groups, any resolution of ``ties'' works, and we get the reflection orderings used in \cite{athanasiadis2007shellability} to prove that $[1,w]$ is EL-shellable.
For affine Coxeter groups, a more sophisticated definition is needed to handle ties between so-called horizontal reflections.
Nevertheless, we hope that the general idea of \emph{axial orderings} can be extended to more general Coxeter groups.

\begin{question}
    Can axial orderings be defined for general Coxeter groups in such a way that the natural labeling of $[1,w]$ is an EL-labeling?
    \label{question:shellability}
\end{question}

A proof of EL-shellability using axial orderings is likely going to need the following compatibility property as an intermediate step.

\begin{question}
    Does an axial ordering for $W$ restrict to axial orderings for all subgroups generated by the reflections below $u \in [1, w]$?
\end{question}

\section{Dual Artin groups and classifying spaces}
\label{sec:classifying-spaces}

As usual, let $W$ be a Coxeter group and $w$ one of its Coxeter elements.
Recall from \Cref{sec:general-picture} that the interval complex $K_W$ of $[1,w]$ is a $\Delta$-complex (as in Hatcher's book \cite{hatcher}) with one $d$-dimensional simplex $[x_1|x_2|\dotsb|x_d]$ for every factorization $u = x_1x_2\dotsm x_d$ such that $l_R(u) = l_R(x_1) + l_R(x_2) + \dotsm + l_R(x_d)$ and $u \in [1, w]$.
In other words, simplices correspond to partial minimal factorizations of $w$.
The faces of $[x_1|x_2|\dotsb|x_d]$ are given by:
\begin{itemize}
    \setlength\itemsep{0.08cm}
    \item $[x_2|\dotsb|x_d]$;
    \item $[x_1|x_2|\dotsb|x_ix_{i+1}|\dotsb|x_d]$ for $i = 1, \dotsc, d-1$;
    \item $[x_1|x_2|\dotsb|x_{d-1}]$.
\end{itemize}
The fundamental group $W_w = \pi_1(K_W)$ is the dual Artin group associated with the Coxeter group $W$ and the Coxeter element $w$.

If $[1,w]$ is a lattice, then $W_w$ is a Garside group and $K_W$ is a classifying space for $W_w$.
Interestingly, it was shown in \cite[Theorem 6.6]{paolini2021proof} that $K_W$ is a classifying space for every affine Coxeter group $W$, even when the lattice property does not hold.
The proof makes use of a construction by McCammond-Sulway \cite{mccammond2017artin} of a crystallographic group $C \supseteq W$ where the interval $[1,w]_{R'}$ (with respect to an extended generating set $R' \supseteq R$) is a lattice.
Therefore, we naturally pose the following two questions.

\begin{question}
    For a general Coxeter group $W$ with a Coxeter element $w$, is it possible to extend the generating set $R$ to a larger set $R'$ of isometries which generate a discrete group $C \supseteq W$ so that (1) the interval $[1,w]_{R'}$ is a lattice; (2) the dual Artin group $W_w$ embeds into the interval group $C_w$?
    \label{question:completion}
\end{question}

\begin{question}
    Is the interval complex $K_W$ a classifying space?
\end{question}

If $K_W$ is a classifying space for the dual Artin group $W_w$, then one can hope to prove the $K(\pi, 1)$ conjecture as in the affine case, by constructing a homotopy equivalence between $K_W$ and the orbit configuration space $Y_W$ (or equivalently, the Salvetti complex $X_W$).
Regardless of whether $K_W$ is a classifying space, a homotopy equivalence $K_W \simeq Y_W$ implies that the dual Artin group $W_w$ is isomorphic to the standard Artin group $G_W$.
Note that there is a natural map $G_W \to W_w$ which sends the standard generators of $G_W$ to the same generators inside $W_w$, and it makes sense to expect this map to be an isomorphism.

\begin{question}
    Is the interval complex $K_W$ homotopy equivalent to the orbit configuration space $Y_W$ (or equivalently, to the Salvetti complex $X_W$)?
    \label{question:homotopy-equivalence}
\end{question}

\begin{question}
    Is any dual Artin group $W_w$ (naturally) isomorphic to the corresponding standard Artin group $G_W$?
    \label{question:isomorphism}
\end{question}

In the spherical case, Bessis proved that $[1,w]$ is always a lattice (\Cref{question:lattice}) and that $W_w$ is isomorphic to $G_W$ (\Cref{question:isomorphism}).\footnote{There is in fact a uniform proof of the isomorphism $W_w \cong G_W$ in the spherical case, as discussed in \cite{chapuy2021counting}.}
Then \Cref{question:homotopy-equivalence} was a consequence of the fact that $K_W$ and $Y_W$ are classifying spaces of isomorphic groups.
In the affine case, \Cref{question:isomorphism} was first settled by McCammond-Sulway \cite{mccammond2017artin} and then re-proved in \cite{paolini2021proof} by answering \Cref{question:homotopy-equivalence}.

In order to attempt a proof that $K_W \simeq X_W$, it is convenient to fix a totally ordered set of simple reflections $S = \{s_1, s_2, \dotsc, s_n\}$ such that $w = s_1 s_2 \dotsm s_n$.
Then, the cells of $X_W$ are indexed by the subsets $T \subseteq S$ such that the (standard) parabolic subgroup $W_T \subseteq W$ is finite.
With this setup, \cite[Section 5]{paolini2021proof} introduces a new CW complex $X_W' \simeq Y_W$ which is naturally included in $K_W$, defined as
\[
    X_W' = \bigcup_{\substack{\phantom{\frac{a}{b}}{T \subseteq S}\phantom{\frac{a}{b}} \\ \text{$W_T$ finite}}} K_{W_T}.
\]
If $W$ is finite, then $X_W' = K_W$ and nothing interesting happens.
In all other cases, $K_W$ has infinitely many cells, whereas $X_W'$ is always a finite subcomplex.
The affine case $\tilde A_2$ is described in \Cref{fig:dual-salvetti}.

\begin{figure}
    \begin{center}
    \begin{tikzpicture}[scale=1.5,
	extended line/.style={shorten >=-#1,shorten <=-#1}, extended line/.default=45cm]
	
	\clip (-2.23, -2.2) rectangle (3.09, 1.2);

	\fill[black!30] ($(0, 0)$) -- ($(0, -1)$) -- ($({0.5*sqrt(3)}, -0.5)$) -- cycle;

	\newcommand{\rows}{7}
	\begin{scope}[black!50]
	\foreach \row in {-\rows, ...,\rows} {
		\draw [extended line] ($\row*({0.5*sqrt(3)}, 0.5)$) -- ($(0,\rows)+\row*({0.5*sqrt(3)}, -0.5)$);
		\draw [extended line] ($\row*(0, 1)$) -- ($({\rows/2*sqrt(3)}, \rows/2)+\row*({-0.5*sqrt(3)}, 0.5)$);
		\draw [extended line] ($\row*(0, 1)$) -- ($({-\rows/2*sqrt(3)}, \rows/2)+\row*({0.5*sqrt(3)}, 0.5)$);
	}
	\end{scope}
	
	\begin{scope}[black, thick]
	\foreach \row in {-1,0} {
		\draw [extended line] ($\row*({0.5*sqrt(3)}, 0.5) + ({0.5*sqrt(3)}, 0.5)$) -- ($(0,\rows)+\row*({0.5*sqrt(3)}, -0.5) + ({0.5*sqrt(3)}, -0.5)$);
		\draw [extended line] ($\row*(0, 1)$) -- ($({\rows/2*sqrt(3)}, \rows/2)+\row*({-0.5*sqrt(3)}, 0.5)$);
		\draw [extended line] ($\row*(0, 1)$) -- ($({-\rows/2*sqrt(3)}, \rows/2)+\row*({0.5*sqrt(3)}, 0.5)$);
	}
	\end{scope}
	
	\draw [extended line, dashed, orange] ($({0.25*sqrt(3)}, 0)$) -- ($({0.25*sqrt(3)}, 1)$);
	
    \node at (0.57,-0.2) {$a$};
    \node at (-0.1,-0.5) {$b$};
    \node at (0.57,-0.8) {$c$};
    \node at (-0.52,-0.9) {$a'$};
    \node at (1,-0.95) {$b'$};
    \node at (-0.52,-0.12) {$c'$};
	\end{tikzpicture}
    \end{center}
	
    \[
    X_W' \, = \quad
    \begin{tikzpicture}[baseline={([yshift=-.5ex]current bounding box.center)}, scale=1.2]
	    \coordinate (w) at (0, 2.4);
	    \coordinate (a) at (-1.5,1.2);
	    \coordinate (b) at (0.5,1.2);
	    \coordinate (c) at (1.5,1.2);
	    \coordinate (1) at (0,0);
	    
	    \begin{scope}[every node/.append style = {font=\small}]
		    \draw (a) -- (w);
		    \draw (1) -- (w) node [midway, pos=0.5, fill=white, inner sep=2.5] {{$ab$}};
		    \draw (1) -- (a);
		    \draw (1) -- (c);
		    \draw (c) -- (w);
		    \draw (1) -- (b) node [midway, pos=0.7, below] {\;\;{$b$}};
		    \draw (b) -- (w) node [midway, pos=0.3, above] {\;\;\;{$c'$}};
	    \end{scope}
	    
 	    \node (a1) at (-0.8, 0.45) {{$a$}};
	    \node (c1) at (0.8, 0.45) {{$c'$}};
	    \node (b2) at (-0.8, 1.95) {{$b$}};
	    \node (a2) at (0.8, 1.95) {{$a$}};

	    \node at (1) {\phantom{$1$}};
    \end{tikzpicture}
    \quad\cup\quad
    \begin{tikzpicture}[baseline={([yshift=-.5ex]current bounding box.center)}, scale=1.2]
	    \coordinate (w) at (0, 2.4);
	    \coordinate (a) at (-1.5,1.2);
	    \coordinate (b) at (0.5,1.2);
	    \coordinate (c) at (1.5,1.2);
	    \coordinate (1) at (0,0);
	    
	    \begin{scope}[every node/.append style = {font=\small}]
		    \draw (a) -- (w);
		    \draw (1) -- (w) node [midway, pos=0.5, fill=white, inner sep=2.5] {{$ac$}};
		    \draw (1) -- (a);
		    \draw (1) -- (c);
		    \draw (c) -- (w);
		    \draw (1) -- (b) node [midway, pos=0.7, below] {\;\;{$c$}};
		    \draw (b) -- (w) node [midway, pos=0.3, above] {\;\;\;{$b'$}};
	    \end{scope}
	    
 	    \node (a1) at (-0.8, 0.45) {{$a$}};
	    \node (c1) at (0.8, 0.45) {{$b'$}};
	    \node (b2) at (-0.8, 1.95) {{$c$}};
	    \node (a2) at (0.8, 1.95) {{$a$}};
	    
	    \node at (1) {\phantom{$1$}};
    \end{tikzpicture}
    \quad\cup\quad
    \begin{tikzpicture}[baseline={([yshift=-.5ex]current bounding box.center)}, scale=1.2]
	    \coordinate (w) at (0, 2.4);
	    \coordinate (a) at (-1.5,1.2);
	    \coordinate (b) at (0.5,1.2);
	    \coordinate (c) at (1.5,1.2);
	    \coordinate (1) at (0,0);
	    
	    \begin{scope}[every node/.append style = {font=\small}]
		    \draw (a) -- (w);
		    \draw (1) -- (w) node [midway, pos=0.5, fill=white, inner sep=2.5] {{$bc$}};
		    \draw (1) -- (a);
		    \draw (1) -- (c);
		    \draw (c) -- (w);
		    \draw (1) -- (b) node [midway, pos=0.7, below] {\;\;{$c$}};
		    \draw (b) -- (w) node [midway, pos=0.3, above] {\;\;\;{$a'$}};
	    \end{scope}
	    
 	    \node (a1) at (-0.8, 0.45) {{$b$}};
	    \node (c1) at (0.8, 0.45) {{$a'$}};
	    \node (b2) at (-0.8, 1.95) {{$c$}};
	    \node (a2) at (0.8, 1.95) {{$b$}};

	    \node at (1) {\phantom{$1$}};
    \end{tikzpicture}
    \]
    \caption{\textbf{Top:} the affine reflection arrangement of type $\tilde A_2$.
    The chamber corresponding to the set of simple reflections $S = \{a,b,c\}$ is highlighted.
    The dashed line is the axis of the Coxeter element $w = abc$.
    \textbf{Bottom:} the complex $X_W'$ in the case $\tilde A_2$. It is the union of three copies of the dual interval complex of type $A_2$ (\Cref{fig:ncp}), one for each standard parabolic subgroup. Overall, $X_W'$ consists of the following simplices: the $0$-simplex $[\,]$; the nine $1$-simplices $[a]$, $[b]$, $[c]$, $[a']$, $[b']$, $[c']$, $[ab]$, $[ac]$, and $[bc]$; the nine $2$-simplices $[a|b]$, $[b|c']$, $[c'|a]$, $[a|c]$, $[c|b']$, $[b'|a]$, $[b|c]$, $[c|a']$, and $[a'|b]$.}
    \label{fig:dual-salvetti}
\end{figure}
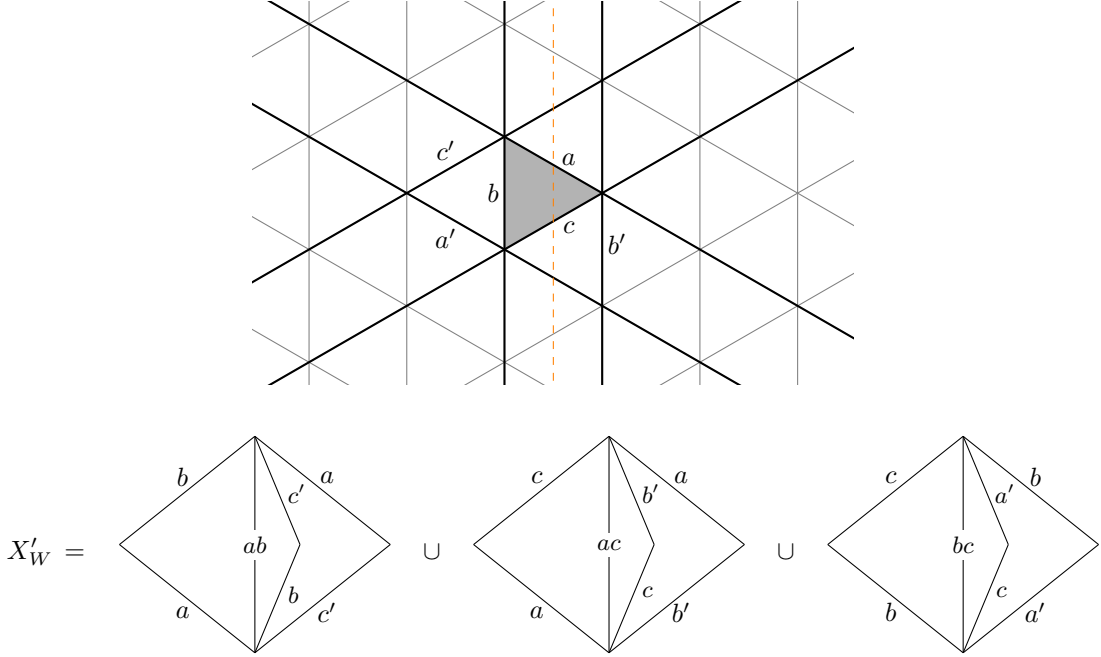

Note that the definition of $X_W'$ requires to fix a Coxeter element $w_T \in W_T \cap [1, w]$ for every $T$, and this is done by multiplying the elements of $T$ in the same order as they appear in $S$.
Roughly speaking, $X_W'$ locally looks like $K_W$ but globally it has the structure of the Salvetti complex $X_W$.
This, together with the known fact that $K_{W_T} \simeq X_{W_T}$ for finite $W_T$, implies that the newly defined $X_W'$ is indeed homotopy equivalent to the Salvetti complex $X_W$.
Note that $X_W'$ is defined and proved to be homotopy equivalent to $X_W$ (and thus to $Y_W$) for general Coxeter groups.
Note also that $X_W'$ depends on the choice of $S$ and not only on the Coxeter element $w$ (as opposed to $K_W$, which depends on $w$ but not on $S$).

The complexes $X_W'$ provide a more direct link between $K_W$ and $Y_W$.
With them, we can strengthen \Cref{question:homotopy-equivalence} while possibly coming closer to its proof.

\begin{question}
    Does the interval complex $K_W$ deformation retract onto any (or all) of its subcomplexes $X_W'$?
\end{question}

In the affine case, this is proved within the framework of discrete Morse theory \cite{forman1998morse, forman2002user}.
The set of simple reflections $S$ used to construct $X_W'$ consists of the reflections with respect to the walls of an axial chamber (see \Cref{question:axial-factorization}).
The ``discrete Morse vector field'' (a.k.a.\ the Morse matching) is constructed in two stages.
First, all but a finite number of the simplices in $K_W \setminus X_W'$ of the form $[x_1|x_2|\dotsb |x_d]$ with $x_1x_2 \dotsm x_d = w$ are matched with either $[x_1|x_2|\dotsb |x_{d-1}]$ or $[x_2|x_3|\dotsb|x_d]$.
This is a natural way to collapse a large number of simplices of $K_W$ and could potentially be useful in other cases.
Second, the (finitely many) remaining simplices of $K_W \setminus X_W'$ are collapsed by using an axial ordering $\prec$ of $R_0$ and the EL-shellability property (see \Cref{question:shellability}).
Roughly speaking, a simplex $\sigma = [x_1|x_2|\dotsb |x_d]$ is matched with the simplex $\tau$ computed through the following procedure:
\begin{enumerate}[1.]
    \item let $i=1$;
    \item if $l(x_i)>1$, then let $\tau$ be the simplex obtained from $\sigma$ by replacing $x_i$ with $r | rx_i$ where $r$ is the $\prec$-smallest reflection below $x_i$;
    \item if $x_i$ is a reflection $\prec$-smaller than all reflections below $x_{i+1}$, then let $\tau$ be the simplex obtained from $\sigma$ by replacing $x_i | x_{i+1}$ with $x_ix_{i+1}$;
    \item otherwise, increase $i$ by $1$ and repeat from step 2.
\end{enumerate}

The uniqueness of $\prec$-increasing factorizations, granted by the EL-shellability property, ensures that the previous procedure actually defines an involution (if $\sigma$ is matched with $\tau$, then $\tau$ is matched with $\sigma$).
For example, in the case $\tilde A_2$ (with the notation of \Cref{fig:dual-salvetti}), the simplex $[w] = [abc]$ is matched with $[a|bc]$ because $a$ is the $\prec$-smallest reflection below $w$.
Conversely, $[a|bc]$ is matched with $[abc] = [w]$ because $a$ is $\prec$-smaller than all reflections below $bc$ (these reflections are $b$, $c$, $a'$).
We also have that $[ab|c]$ is matched with $[a|b|c]$.
This procedure could be used more generally, provided that the EL-shellability property holds.

The lattice property for $[1,w]$ (\Cref{question:lattice}) implies that the dual Artin group $W_w$ is a Garside group and thus it has a solvable word problem, provided that one can effectively compute meets and joins in $[1,w]$.
Together with a constructive proof of the isomorphism $W_w \cong G_W$ (\Cref{question:isomorphism}), this implies that the word problem for the Artin group $G_W$ is also solvable.
When the lattice property does not hold, then a ``completion'' $C_w \supseteq W_w$ as in \cite{mccammond2017artin} (\Cref{question:completion}) could be used to solve the word problem for $W_w$ by leveraging a solution to the word problem for the Garside group $C_w$.

\begin{figure}[t]
    \begin{center}
    \begin{tikzpicture}
        \node[draw] (41) at (-1,0) {Q4.1: Lattice property};
        \node[draw] (42) at (4,0) {Q4.2: Shellability};
        \node[draw, text width=2.7cm, align=center] (43) at (8,0) {Q4.3: Restriction\\of axial orderings};
      
        \node[draw, text width=3.5cm, align=center] (51) at (-1,-2) {Q5.1: Garside\\completion $C_w \supseteq W_w$};

        \node[draw, text width=2.65cm, align=center] (52) at (4,-2) {Q5.2: $K_W$\\classifying space};

        \node[draw] (55) at (8, -2) {Q5.5: $K_W \searrow X_W'$};
        \node[draw] (53) at (8, -3.5) {Q5.3: $K_W \simeq Y_W$};
        \node[draw] (54) at (8, -5.5) {Q5.4: $W_w \cong G_W$};

        \node[draw, very thick] (conjecture) at (1.5, -4.5) {$K(\pi, 1)$ conjecture};
        \node[draw, very thick] (wordproblem) at (1.5, -6.5) {Word problem for $G_W$};
        
        \draw[-Latex] (41) to (51);
        \draw[-Latex] (41) to[in=90, out=270, looseness=0.5] (52);
        \draw[-Latex, dashed] (51) to (52);

        \draw[-Latex, dashed] (43) to (42);
        \draw[-Latex, dashed] (42) to[in=90, out=270, looseness=0.5] (55);
        \draw[-Latex] (55) to (53);
        \draw[-Latex] (53) to (54);
        
        \node[inner sep=0] (A) at (3.4, -3.5) {};
        \draw (52) to[in=0, out=270, looseness=0.5] (A.center);
        \draw (53) to[in=0, out=180, looseness=0.5] (A.center);
        \draw[-Latex] (A.center) to[in=90, out=180, looseness=1] (conjecture);

        \node[inner sep=0] (B) at (1.5, -5.9) {};
        \draw (54) to[in=90, out=180, looseness=0.3] (B.center);
        \draw (51) -- (-1, -4.8) to[in=90, out=270, looseness=0.7] (B.center);
        \draw[-Latex] (B.center) -> (wordproblem);

    \end{tikzpicture}
    \end{center}
    \caption{Diagram of the questions from \Cref{sec:combinatorics,sec:classifying-spaces}.
    Solid arrows
    mean ``implies'', whereas dashed arrows mean ``could be helpful for''.
    }
    \label{fig:diagram}
\end{figure}

\Cref{fig:diagram} summarizes all questions from this and the previous section, as well as the implications between them and with the $K(\pi, 1)$ conjecture and the word problem.

\section{Beyond spherical and affine cases}

The questions posed in this paper are answered for spherical and affine Coxeter/Artin groups but mostly remain mysterious beyond those cases.
Bessis proved that $[1,w]$ is a lattice if $W$ is a \emph{universal Coxeter group} (all the labels in the Coxeter graph are $\infty$), thus positively answering \Cref{question:lattice} in this case \cite{bessis2006dual}.
In an upcoming work with Emanuele Delucchi and Mario Salvetti, we are going to completely address Coxeter groups of rank $3$ (i.e., with $3$ generators):

\begin{theorem}[Delucchi-Paolini-Salvetti \cite{delucchi2022dual}]
    All questions in this paper have a positive answer for Coxeter groups of rank $3$.
\end{theorem}

Examples of such groups are given in \Cref{fig:triangle-groups}.
Note that the $K(\pi, 1)$ conjecture was already known if $W$ has rank $3$ because $W$ is either spherical or $2$-dimensional.
However, most of the other questions are not trivial and require a study of the geometry and combinatorics of the dual structure.
Hopefully, this is going to be a useful step towards a better understanding of the dual approach for general Coxeter groups.

\bibliographystyle{amsalpha-abbr}
\bibliography{bibliography}

\end{document}